\documentclass[11pt]{article}
\usepackage{enumerate}
\usepackage{amssymb,a4wide,latexsym,makeidx,epsfig,fleqn}
\usepackage{amsthm}
\usepackage{amsmath}
\usepackage{enumerate}
\usepackage{graphicx}
\usepackage{float}
\usepackage{tikz}
\usepackage[colorlinks=true, linkcolor=blue, citecolor=red, urlcolor=blue]{hyperref}
\allowdisplaybreaks[4]
\newtheorem{theorem}{Theorem}[section]

\newtheorem{lemma}[theorem]{Lemma}

\newtheorem{corollary}[theorem]{Corollary}

\begin{document}
\textwidth 150mm \textheight 225mm
\title{Laplacian eigenvalue conditions for edge-disjoint spanning trees and a forest with constraints\thanks{Supported by the National Natural Science Foundation of China (No. 12271439).}}
\author{{Yongbin Gao$^{a,b}$, Ligong Wang$^{a,b,}$\footnote{Corresponding author.}}\\
{\small $^a$ School of Mathematics and Statistics, Northwestern
Polytechnical University,}\\ {\small  Xi'an, Shaanxi 710129,
P.R. China.}\\
{\small $^b$ Xi'an-Budapest Joint Research Center for Combinatorics, Northwestern
Polytechnical University,}\\
{\small Xi'an, Shaanxi 710129,
P.R. China. }\\
{\small E-mail: gybmath@163.com, lgwangmath@163.com} }
\date{}
\maketitle
\begin{center}
\begin{minipage}{120mm}
\vskip 0.3cm
\begin{center}
{\small {\bf Abstract}}
\end{center}
{\small Let $k$ be a positive integer and let $G$ be a simple graph of order $n$ with minimum degree $\delta$. A graph $G$ is said to have property $P(k, d)$ if it contains $k$ edge-disjoint spanning trees and an additional forest $F$ with edge number $|E(F)| > \frac{d-1}{d}(|V(G)| - 1)$, such that if $F$ is not a spanning tree, then $F$ has a component with at least $d$ edges. Let $D(G)$ be the degree diagonal matrix of $G$. We denote $\lambda_i$ and $\mu_i$ as the $i$th largest eigenvalue of the adjacency matrix $A(G)$ of $G$ and the Laplacian matrix $L(G) = D(G) - A(G)$ of $G$ for $i = 1, 2, \ldots, n$, respectively.

In this paper, we investigate the relationship between Laplacian eigenvalues and property $P(k, \delta)$. Let $t$ be a positive integer, and define $\mathcal{G}_t$ as the set of simple graphs such that each $G \in \mathcal{G}_t$ contains at least $t+1$ non-empty disjoint proper subsets $V_1, V_2, \ldots, V_{t+1}$ satisfying $V(G) \setminus \bigcup_{i=1}^{t+1} V_i \neq \emptyset$ and edge connectivity $\kappa'(G) = e(V_i, V(G) \setminus V_i)$ for any $i = 1, 2, \ldots, t+1$. For the class of graphs $\mathcal{G}_1$ with minimum degree $\delta$, we provide a sufficient condition involving the third smallest Laplacian eigenvalue $\mu_{n-2}(G)$ for a graph $G\in \mathcal{G}_1$ to have property $P(k, \delta)$. Similarly, for the class of graphs $\mathcal{G}_2$ with minimum degree $\delta$, we establish a corresponding sufficient condition involving the fourth smallest Laplacian eigenvalue $\mu_{n-3}(G)$ for a graph $G\in \mathcal{G}_2$ to have property $P(k, \delta)$. Furthermore, we extend the spectral conditions for all the
results about $\mu_{n-2}(G)$, $\mu_{n-3}(G)$ and $\lambda_2(G)$ to the general graph matrices $aD(G) + A(G)$ and $aD(G) + bA(G)$.

 }
\end{minipage}
\end{center}

\section{Introduction}
In this paper, we consider only finite, undirected, and simple graphs. Let $G$ be a graph with vertex set $V(G) = \{v_1, v_2, \ldots, v_n\}$ and edge set $E(G) = \{e_1, e_2, \ldots, e_m\}$. We use $\delta = \delta(G)$ and $\Delta = \Delta(G)$ to denote the minimum degree and maximum degree of $G$, respectively. For a connected graph $G$, let $\tau(G)$ denote the maximum number of edge-disjoint spanning trees contained in $G$. This parameter is commonly referred to as the spanning tree packing number.

Let $X$ and $Y$ be two disjoint subsets of $V(G)$. We denote by $E(X, Y)$ the set of edges in $G$ with one vertex in $X$ and the other vertex in $Y$, and let $e(X, Y) = |E(X, Y)|$. Let $t$ be a positive integer. Define $\mathcal{G}_t$ as the set of simple graphs such that each $G \in \mathcal{G}_t$ contains at least $t+1$ non-empty disjoint proper subsets $V_1, V_2, \ldots, V_{t+1}$ satisfying $V(G) \setminus \bigcup_{i=1}^{t+1} V_i \neq \emptyset$ and edge connectivity $\kappa'(G) = e(V_i, V(G) \setminus V_i)$ for any $i = 1, 2, \ldots, t+1$.

The adjacency matrix of a graph $G$ of order $n$, denoted by $A(G)$, is an $n \times n$ matrix whose $(i,j)$-entry is $1$ if $v_i$ and $v_j$ are adjacent, and $0$ otherwise. Let $D(G) = \text{diag}(d_1, d_2, \ldots, d_n)$ be the degree diagonal matrix of $G$, where $d_i$ is the degree of vertex $v_i$. The Laplacian matrix of $G$ is defined as $L(G) = D(G) - A(G)$ and the signless Laplacian matrix of $G$ is defined as $Q(G) = D(G) + A(G)$. Since $A(G)$, $L(G)$ and $Q(G)$ are real and symmetric, all their eigenvalues are real. Throughout this paper, we denote the $i$th largest eigenvalue of $A(G)$, $L(G)$ and $Q(G)$ by
$\lambda_i(G) $, $ \mu_i(G) $ and $ q_i(G)$, respectively.

The problem of determining the relationship between $\tau(G)$ and the eigenvalues of $G$ was originally proposed by Seymour in a private communication to Cioab\v{a} (see $\cite{RF1}$). Motivated by this problem, Cioab\v{a} and Wong $\cite{RF1}$ initiated the study of $\tau(G)$ using the second largest eigenvalue $\lambda_2(G)$. They provided sufficient spectral conditions for $\tau(G) \ge k$ in $d$-regular graphs for $k=2, 3$, and conjectured a general bound for $4 \le k \le \lfloor\frac{d}{2}\rfloor$. This conjecture was subsequently generalized to non-regular graphs by Gu et al. $\cite{RF2}$ and was eventually fully resolved by Liu et al. $\cite{RF3, RF13}$ in all cases. Cioab\v{a} et al. $\cite{RF15}$ later constructed extremal graphs to demonstrate that the bound obtained in $\cite{RF3}$ is essentially tight. Liu et al. \cite{RF16} improved the result by considering the girth of the graph.

In addition to $\lambda_{2}(G)$, other eigenvalues have also attracted attention. Fan et al. $\cite{RF4}$ investigated the relationship between the spectral radius $\rho(G)$(or $\lambda_1(G)$) and $\tau(G)$ by analyzing extremal graphs. Duan et al. $\cite{RF5}$ extended the scope to the third largest eigenvalue, establishing connections between $\tau(G)$ and $\lambda_3(G)$ (as well as $\mu_{n-2}(G)$ and $q_3(G)$) for graphs in $\mathcal{G}_1$. Hu et al. $\cite{RF6}$ determined the relationship between $\tau(G)$ and $\mu_{n-2}(G)$ for graphs in $\mathcal{G}_1$, as well as the relationship between $\tau(G)$ and $\mu_{n-3}(G)$ for graphs in $\mathcal{G}_2$.

An important tool in these investigations is the classical Tree Packing Theorem by Nash-Williams $\cite{RF7}$ and Tutte $\cite{RF8}$. Let $\mathcal{P} = \{V_1, V_2, \ldots, V_p\}$ be a partition of the vertex set $V(G)$ with size $|\mathcal{P}| = p \ge 2$, where $p$ denotes the number of parts in the partition. For any nontrivial graph $G$, the fractional packing number $\nu_f(G)$ is defined as
\[
	\nu_f(G) = \min_{p \ge 2} \frac{\sum_{1 \le i < j \le p} e(V_i, V_j)}{p - 1}.
\]
The classical Tree Packing Theorem states that $\tau(G) \ge k$ if and only if $\nu_f(G) \ge k$. 

Spectral conditions for the fractional packing number were earlier studied by Hong et al. \cite{RF14}. Recently, Fang and Yang $\cite{RF9}$ provided a structural explanation for the fractional part of $\nu_f(G)$. This motivated Cai and Zhou $\cite{RF10}$ to define the property $P(k, d)$. A graph $G$ is said to have property $P(k, d)$ if it satisfies the following:
\begin{enumerate}[{\rm(a)}]
	\item $\tau(G) \ge k$,
	\item apart from $k$ edge-disjoint spanning trees, there exists another forest $F$ with size $|E(F)| > \frac{d-1}{d}(n-1)$,
	\item if $F$ is not a spanning tree, then $F$ has a component with at least $d$ edges.
\end{enumerate}

The following result by Fang and Yang $\cite{RF9}$ (see also $\cite{RF10}$) serves as a bridge between the fractional packing number and property $P(k, d)$, which plays a crucial role in our proofs.
\noindent\begin{theorem}[$\cite{RF10, RF9}$] \label{thm:1.1}
	Let $k$ and $d$ be positive integers. For a nontrivial graph $G$, if \[\nu_f(G) > k + \frac{d-1}{d},\] then $G$ has property $P(k, d)$.
\end{theorem}
Naturally, this leads to a fundamental problem: what is the relationship between the eigenvalues of a graph and property $P(k,d)$? Cai and Zhou $\cite{RF10}$ started this investigation by establishing connections between property $P(k,d)$ and the spectral radius $\rho(G)$ (or $\lambda_1(G)$), as well as the second largest adjacency eigenvalue $\lambda_2(G)$. 

Inspired by the works of Cai and Zhou $\cite{RF10}$ and Hu et al. $\cite{RF6}$, this paper primarily investigates the relationship between Laplacian eigenvalues and property $P(k,d)$. In Section 3, we establish a sufficient condition involving the third smallest Laplacian eigenvalue $\mu_{n-2}(G)$ for graphs in $\mathcal{G}_1$. In Section 4, we extend our analysis to the class of graphs $\mathcal{G}_2$, providing a condition involving the fourth smallest Laplacian eigenvalue $\mu_{n-3}(G)$. Furthermore, we generalize these main results to the spectra of matrices $aD(G) + A(G)$ and $aD(G) + bA(G)$, where $a$ and $b$ are real numbers satisfying $a \ge -1$, $b \neq 0$, and $\frac{a}{b} \ge -1$. The main results relating Laplacian eigenvalues to property $P(k, \delta)$ are stated in the following theorems.
\noindent\begin{theorem} \label{thm:1.2}
	Let $k \ge 2$ be an integer and let $G \in \mathcal{G}_1$ be a graph with minimum degree $\delta \ge 2k+2$. If
	\[
	\mu_{n-2}(G) > \frac{16(k+\frac{\delta-1}{\delta})}{3(\delta+1)},
	\]
	then $G$ has property $P(k, \delta)$.
\end{theorem}
\noindent\begin{theorem} \label{thm:1.3}
	Let $k \ge 2$ be an integer and let $G \in \mathcal{G}_2$ be a graph with minimum degree $\delta \ge 3k+3$. If
	\[
	\mu_{n-3}(G) > \frac{9(k+\frac{\delta-1}{\delta})}{\delta+1},
	\]
	then $G$ has property $P(k, \delta)$.
\end{theorem}
Additionally, as a supplement, in Section 5 we extend the results of Cai and Zhou $\cite{RF10}$ regarding $\lambda_2(G)$ to the matrices of the forms $aD(G) + A(G)$ and $aD(G) + bA(G)$, where $a$ and $b$ are real numbers satisfying $a\ge 0$ and $b \neq 0$.

\section{Preliminaries}
In this section, we present some of the preliminaries and known results to be used in this paper.

Let $B$ be a real symmetric matrix of order $n$. Suppose that the rows and columns of $B$ are partitioned according to a partition $\pi = \{X_1, X_2, \ldots, X_t\}$ of the index set $\{1, 2, \ldots, n\}$. The quotient matrix of $B$ with respect to $\pi$, denoted by $B_{\pi}$, is the $t \times t$ matrix whose $(i,j)$-entry is the average row sum of the block $B_{ij}$ formed by the intersection of rows in $X_i$ and columns in $X_j$.

Let $V(G)$ be the vertex set of a graph $G$. Suppose that $V(G)$ is partitioned into $t$ non-empty subsets $V_1, V_2, \ldots, V_t$. Let $M_t$ denote the quotient matrix of the Laplacian matrix $L(G)$ with respect to this partition. The entries $m_{ij}$ of $M_t$ are given by
\[
	m_{ij} = \begin{cases}
	\sum_{1\le k \neq i \le t } \frac{e(V_i, V_k)}{|V_i|}, & \text{if } i = j, \\
	-\frac{e(V_i, V_j)}{|V_i|}, & \text{if } i \neq j.
\end{cases}
\] 
Recall that the sum of eigenvalues of a matrix equals its trace. Thus, for the quotient matrix $M_t$, we have
\[ 
	\sum_{i=1}^{t} \lambda_i(M_t) = \text{tr}(M_t),
\]
where $\lambda_1(M_t) \ge \dots \ge \lambda_t(M_t)$ are the eigenvalues of $M_t$.

The relationship between the eigenvalues of a matrix and its quotient matrix is given by the celebrated Cauchy Interlacing Theorem. Before stating it, we define the concept of interlacing.

Consider two sequences of real numbers $\theta_1 \ge \theta_2 \ge \dots \ge \theta_n$ and $\eta_1 \ge \eta_2 \ge \dots \ge \eta_m$ with $m < n$. The second sequence is said to \emph{interlace} the first one if
\[ 
\theta_i \ge \eta_i \ge \theta_{n-m+i}, \quad \text{for } i=1, 2, \ldots, m. 
\]
\noindent\begin{theorem}[Cauchy Interlacing Theorem \cite{RF11,RF12}]\label{thm:2.1} 
	Let $M$ be a real symmetric matrix. Then the eigenvalues of every quotient matrix of $M$ interlace the eigenvalues of $M$. 
\end{theorem}
\noindent\begin{theorem}[Weyl's Inequalities $\cite{RF11}$]\label{thm:2.2} 
	Let $A$ and $B$ be Hermitian matrices of order $n$, and let $1 \le i, j \le n$.
	\begin{enumerate}[{\rm(i)}]
		\item If $i + j - 1 \le n$, then
		$\lambda_{i+j-1}(A + B) \le \lambda_i(A) + \lambda_j(B)$;
		\item If $i + j - n \ge 1$, then
		$\lambda_i(A) + \lambda_j(B) \le \lambda_{i+j-n}(A + B)$.
	\end{enumerate}
\end{theorem}
\noindent\begin{lemma}[Lemma 2.8 in $\cite{RF2}$]\label{lem:2.3} 
	Let $G$ be a connected graph with minimum degree $\delta$, and let $U$ be a non-empty proper subset of $V(G)$. If $e(U, V \setminus U) \le \delta - 1$, then $|U| \ge \delta + 1$.
\end{lemma}
The following lemma is obtained from Theorem 3.4 in \cite{RF6} by replacing the parameter $k$ with $k+1$.
\noindent\begin{lemma}[Adapted from Theorem 3.4 in $\cite{RF6}$]\label{lem:2.4} 
	 Let $G \in \mathcal{G}_1$ be a graph with minimum degree $\delta \ge 2k+1 \ge 5$. Suppose that $X \subseteq E(G)$ is an edge subset such that $G-X$ is disconnected. Let $G'$ be any connected component of $G-X$. If $\mu_{n-2}(G) > \frac{4k}{\delta+1}$, then $e(V(G'), V(G) \setminus V(G')) \ge k+1$.
\end{lemma}
Similarly, the next lemma is obtained from Theorem 4.4 in \cite{RF6} by substituting the parameter $k$ with $k+1$.
\noindent\begin{lemma}[Adapted from Theorem 4.4 in $\cite{RF6}$]\label{lem:2.5} 
	Let $G \in \mathcal{G}_2$ be a graph with minimum degree $\delta \ge 3k+1 \ge 7$. Suppose that $X \subseteq E(G)$ is an edge subset such that $G-X$ is disconnected. Let $G'$ be any connected component of $G-X$. If $\mu_{n-3}(G) > \frac{6k}{\delta+1}$, then $e(V(G'), V(G) \setminus V(G')) \ge k+1$.
\end{lemma}

\section{Laplacian eigenvalue $\mu_{n-2}(G)$ and property $P(k,\delta)$ for graphs in  $\mathcal{G}_1$}
In this section, we provide the proof of Theorem \ref{thm:1.2}, which establishes a sufficient condition for a graph $G \in \mathcal{G}_1$ to have property $P(k, \delta)$ in terms of the third smallest Laplacian eigenvalue $\mu_{n-2}(G)$.

\begin{proof}[Proof of Theorem \ref{thm:1.2}]
	For a contradiction, suppose that $G$ does not have property $P(k, \delta)$. By Theorem \ref{thm:1.1}, this implies that the fractional packing number satisfies
	\[ 
	\nu_f(G) \le k + \frac{\delta - 1}{\delta}. 
	\]
	By the definition of $\nu_f(G)$, there exists some partition $\mathcal{P} = \{V_1, V_2, \ldots, V_s\}$ of $V(G)$ with $s \ge 2$ such that
	\begin{equation} \label{eq:1}
		\sum_{1 \le i < j \le s} e(V_i, V_j) \le \left( k + \frac{\delta - 1}{\delta} \right) (s - 1).
	\end{equation}
	
	Let $r_i = e(V_i, V(G) \setminus V_i)$ for $i = 1, 2, \ldots, s$. Note that the set of edges $\bigcup_{1 \le i < j \le s} E(V_i, V_j)$ forms an edge cut of $G$ that separates $G$ into components containing $G[V_1], G[V_2], \ldots, G[V_s]$. Since $\delta \ge 2k+2$ and $\mu_{n-2}(G) > \frac{16(k+\frac{\delta-1}{\delta})}{3(\delta+1)}$, we have $\mu_{n-2}(G) > \frac{4k}{\delta+1}$. 
	Thus, the condition of Lemma \ref{lem:2.4} is satisfied. Applying Lemma \ref{lem:2.4}, we obtain $r_i \ge k+1$ for all $i = 1, 2, \ldots, s$.
	
	Furthermore, by $\sum_{i=1}^s r_i = 2 \sum_{1 \le i < j \le s} e(V_i, V_j)$ and Inequality \eqref{eq:1}, we have
	\begin{equation} \label{eq:2}
		\sum_{i=1}^{s} r_i \le 2 \left( k + \frac{\delta - 1}{\delta} \right) (s - 1).
	\end{equation}
	Without loss of generality, assume that $r_1 \le r_2 \le \cdots \le r_s$.
	
	Let $t$ be the largest integer such that $r_t < 2 \left( k + \frac{\delta - 1}{\delta} \right).$ We claim that $t \ge 3$.
	
	Indeed, if $t=1$, then 
	\[
	\sum_{i=1}^{s}r_i=r_1+\sum_{i=2}^{s}r_i\ge r_1+2(k+\frac{\delta-1}{\delta})(s-1)>2(k+\frac{\delta-1}{\delta})(s-1),
	\]
	which contradicts Inequality \eqref{eq:1}.
	
	If $t=2$, then 
	\[
	\begin{aligned}
		\sum_{i=1}^{s} r_i
		&= r_1 + r_2 + \sum_{i=3}^{s} r_i \\
		&\ge r_1 + r_2 + 2\!\left(k+\frac{\delta-1}{\delta}\right)(s-2) \\
		&\ge 2(k+1) + 2\!\left(k+\frac{\delta-1}{\delta}\right)(s-2)
		> 2\!\left(k+\frac{\delta-1}{\delta}\right)(s-1),
	\end{aligned}
	\]
	which contradicts Inequality \eqref{eq:1} again. Thus, $t \ge 3$ must hold.

	Since $r_i \ge r_2$ for $2 \le i \le t$ and $r_i \ge 2(k+\frac{\delta-1}{\delta})$ for $i > t$, we obtain
	\[ 
	r_1 + (t-1)r_2 + 2 \left( k + \frac{\delta - 1}{\delta} \right)(s - t) \le \sum_{i=1}^s r_i \le 2 \left( k + \frac{\delta - 1}{\delta} \right)(s - 1). 
	\]
	Using the inequality $r_2 \ge \frac{r_1+r_2}{2}$, we have
	\[ 
	r_1 + r_2 \le 2 \left( k + \frac{\delta - 1}{\delta} \right)(t - 1) - (t-2)r_2 \le 2 \left( k + \frac{\delta - 1}{\delta} \right)(t - 1) - (t-2)\frac{r_1+r_2}{2}. 
	\]
	It follows that
	\begin{equation} \label{eq:3}
		r_1 + r_2 \le \frac{4(t-1)}{t} \left( k + \frac{\delta - 1}{\delta} \right).
	\end{equation}
		
	Since $\delta \ge 2k+2$, we have $r_t < 2 \left( k + \frac{\delta - 1}{\delta} \right) < \delta.$ This implies $r_i \le \delta - 1$ for all $i = 1, 2, \ldots, t$. By Lemma \ref{lem:2.3}, it follows that $|V_i| \ge \delta + 1$ for $i = 1, 2, \ldots, t$. Let $V' = V(G) \setminus (V_1 \cup V_2)$. Then $|V'| \ge \sum_{i=3}^t |V_i| \ge (t-2)(\delta+1).$
	
	Let $r_{12} = e(V_1, V_2)$ and $r'_i = e(V_i, V')$ for $i=1, 2$. Consider the quotient matrix $M_3$ of the Laplacian matrix $L(G)$ with respect to the partition $(V_1, V_2, V')$, which is given by
	\[
	M_3 = 
	\begin{pmatrix}
		\frac{r_{1}}{|V_1|} & -\frac{r_{12}}{|V_1|} & -\frac{r'_{1}}{|V_1|} \\
		-\frac{r_{12}}{|V_2|} & \frac{r_{2}}{|V_2|} & -\frac{r'_{2}}{|V_2|} \\
		-\frac{r'_{1}}{|V'|} & -\frac{r'_{2}}{|V'|} & \frac{r'_{1}+r'_{2}}{|V'|}
	\end{pmatrix}.
	\]
	By Theorem \ref{thm:2.1}, we have $\lambda_i(M_3) \ge \mu_{n-3+i}(G) \ge 0$ for $i=1,2,3$.
	By the trace property $\text{tr}(M_3) = \sum_{i=1}^3 \lambda_i(M_3)$, together with Inequality \eqref{eq:3}, we obtain
	\begin{align*}
		\lambda_1(M_3) \le \text{tr}(M_3)
		&= \frac{r_1}{|V_1|} + \frac{r_2}{|V_2|} + \frac{r'_1+r'_2}{|V'|} \\
		&\le \frac{r_1+r_2}{\delta+1} + \frac{r_1+r_2}{(t-2)(\delta+1)} \\
		&= \frac{t-1}{t-2}\cdot\frac{(r_1+r_2)}{\delta+1} \\
		&\le \frac{4(t-1)^2}{t(t-2)}\cdot\frac{(k+\frac{\delta-1}{\delta})}{\delta+1}.
	\end{align*}
	Since $t \ge 3$, the function $f(t) = \frac{(t-1)^2}{t(t-2)}$ is decreasing, and thus bounded by $f(3) = \frac{4}{3}$. Then we obtain
	\[
	\mu_{n-2}(G) \le \lambda_1(M_3) \le \frac{4}{3} \cdot \frac{4(k+\frac{\delta-1}{\delta})}{\delta+1} = \frac{16(k+\frac{\delta-1}{\delta})}{3(\delta+1)},
	\] 
	which is contrary to the assumption in Theorem \ref{thm:1.2}. This completes the proof.
\end{proof}

Having established the relationship between $\mu_{n-2}(G)$ and property $P(k, \delta)$, it is natural to extend this result to more general graph matrices. Let $a$ and $b$ be two real numbers such that $a \ge -1$, $b \neq 0$, and $\frac{a}{b} \ge -1$. In the following corollary, we extend the conclusion of Theorem \ref{thm:1.2} to the spectra of the matrices $aD(G) + A(G)$ and $aD(G) + bA(G)$. Let $\lambda_i(G, a)$ and $\lambda_i(G, a, b)$ denote the $i$-th largest eigenvalue of the matrix $aD(G) + A(G)$ and $aD(G) + bA(G)$, respectively. 
\noindent\begin{corollary} \label{cor:3.1}
	Let $k \ge 2$ be an integer and let $G \in \mathcal{G}_1$ be a graph with minimum degree $\delta \ge 2k+2$. Then $G$ has property $P(k, \delta)$ if one of the following conditions holds:
	\begin{enumerate}[{\rm(i)}]
		\item $\lambda_3(G, a) < (a+1)\delta - \frac{16(k+\frac{\delta-1}{\delta})}{3(\delta+1)}$;
		\item If $b > 0$, $\lambda_3(G, a, b) < (a+b)\delta - \frac{16b(k+\frac{\delta-1}{\delta})}{3(\delta+1)}$;
		\item If $b < 0$, $\lambda_{n-2}(G, a, b) > (a+b)\delta - \frac{16b(k+\frac{\delta-1}{\delta})}{3(\delta+1)}$.
	\end{enumerate}
\end{corollary}
\begin{proof}
	For a contradiction, suppose that $G$ does not have property $P(k, \delta)$. By Theorem \ref{thm:1.2}, we have $\mu_{n-2}(G) \le \frac{16(k+\frac{\delta-1}{\delta})}{3(\delta+1)}$.
	
	Note that $L(G) = D(G) - A(G)$, so $aD(G) + A(G) = (a+1)D(G) - L(G)$. By Weyl's Inequalities (Theorem \ref{thm:2.2}(ii)), we have
	\begin{align*}
		\lambda_3(G, a) 
		&\ge \lambda_n((a+1)D(G)) + \lambda_3(-L(G)) \\
		&= (a+1)\delta - \mu_{n-2}(G) \\
		&\ge (a+1)\delta - \frac{16(k+\frac{\delta-1}{\delta})}{3(\delta+1)},
	\end{align*}
	which contradicts condition (i). Thus, (i) holds.
	
	Note that $aD(G) + bA(G) = b(\frac{a}{b}D(G) + A(G))$.
	If $b > 0$, then $\lambda_i(G, a, b) = b\lambda_i(G, \frac{a}{b})$. Condition (ii) then follows directly from (i) by substituting $a$ with $\frac{a}{b}$ and multiplying by $b$.
	If $b < 0$, then the order of eigenvalues is reversed, so $\lambda_{n-i+1}(G, a, b) = b\lambda_i(G, \frac{a}{b})$. Specifically, for $i=3$, we have $\lambda_{n-2}(G, a, b) = b\lambda_3(G, \frac{a}{b})$. Condition (iii) follows similarly.
\end{proof}

Note that $\lambda_3(G, 0, 1) = \lambda_3(G)$ and $\lambda_3(G, 1, 1) = q_3(G)$. So we obtain the following corollary by corollary \ref{cor:3.1} directly.
\noindent\begin{corollary}\label{cor:3.2} 
	Let $k \ge 2$ be an integer and let $G \in \mathcal{G}_1$ be a graph with minimum degree $\delta \ge 2k+2$. The following statements hold:
	\begin{enumerate}[{\rm(i)}]
		\item If $\lambda_{3}(G)<\delta-\frac{16(k+\frac{\delta-1}{\delta})}{3(\delta+1)}$, then $G$ has property $P(k, \delta)$;
		\item If $q_{3}(G)<2\delta-\frac{16(k+\frac{\delta-1}{\delta})}{3(\delta+1)}$, then $G$ has property $P(k, \delta)$.
	\end{enumerate}
\end{corollary}

\section{Laplacian eigenvalue $\mu_{n-3}(G)$ and property $P(k,\delta)$ for graphs in  $\mathcal{G}_2$}
In this section, we provide the proof of Theorem \ref{thm:1.3}, which establishes a sufficient condition for a graph $G \in \mathcal{G}_2$ to have property $P(k, \delta)$ in terms of the fourth smallest Laplacian eigenvalue $\mu_{n-3}(G)$. 

To handle the case where the partition size is small, we require a structural lemma that allows us to further decompose the graph. The following lemma is an adaptation of Theorem 4.5 in \cite{RF6}, with parameters adjusted to fit our condition for property $P(k,\delta)$.
\begin{lemma}[Adapted from Theorem 4.5 in \cite{RF6}] \label{lem:4.1}
	Let $k \ge 2$ be an integer and $G \in \mathcal{G}_2$ be a connected graph with minimum degree $\delta \ge 3k + 3$. Suppose $X \subseteq E(G)$ is an edge cut such that $G-X$ has exactly 3 components $G_1, G_2, G_3$. Let $r_i = e(V(G_i), V(G) \setminus V(G_i))$. If $\sum_{i=1}^3 r_i \le 4k + 3$ and $k + 1 \le r_i \le 2k + 1$ for each $i$, then there exists an additional edge subset $X' \subseteq E(G) \setminus X$ such that $G - (X \cup X')$ has 4 connected components, $|X'|\le \kappa'(G)$ and each component of $G - (X \cup X')$ has at least $\delta + 1$ vertices.
\end{lemma}
\begin{proof}
	Since $G \in \mathcal{G}_2$, there exist three non-empty disjoint proper subsets $V_1, V_2, V_3$ satisfying $V(G) \setminus (V_1 \cup V_2 \cup V_3) \neq \emptyset$ and $\kappa'(G) = e(V_i, V(G) \setminus V_i)$ for $i=1,2,3$. Let $V' = V(G) \setminus (V_1 \cup V_2 \cup V_3)$. Since $r_i = e(V(G_i), V(G) \setminus V(G_i))$ for $i=1,2,3$, we have $\kappa'(G) \le \min\{r_1, r_2, r_3\}$.
	
	Obviously, there exists a component $G_i$ ($i=1,2,3$) such that $V(G_i)$ has non-empty intersections with at least two sets from $\{V_1, V_2, V_3, V'\}$. Without loss of generality, we assume that $G_1$ has a proper subset $V'_1 = V(G_1) \cap V_1 \neq \emptyset$. Let $X_1 = E(V'_1, V(G_1) \setminus V'_1)$. Then $c(G_1 - X_1) \ge 2$ and $|X_1| \le \kappa'(G)$. Thus, there exists an edge subset $X' \subseteq X_1$ satisfying $c(G_1 - X') = 2$. Let $G_1^1, G_1^2$ be the components of $G_1 - X'$. Then $c(G - X - X') = 4$ and the components of $G - X - X'$ are $G_1^1, G_1^2, G_2, G_3$.
	
	Since $r_i \le 2k+1 \le \delta-1$, we have $|V(G_i)| \ge \delta+1$ for $i=1,2,3$ by Lemma \ref{lem:2.3}. For $G_1^i$ ($i=1,2$), we have
	\begin{align*}
		e(V(G_1^i), V(G) \setminus V(G_1^i)) &= e(V(G_1^i), V(G_1) \setminus V(G_1^i)) + e(V(G_1^i), V \setminus V(G_1)) \\
		&\le |X'| + \max\{r_1, r_2, r_3\} \\
		&\le \kappa'(G) + \max\{r_1, r_2, r_3\} \\
		&\le \min\{r_1, r_2, r_3\} + \max\{r_1, r_2, r_3\} \\
		&\le \sum_{i=1}^3 r_i - \min\{r_1, r_2, r_3\} \\
		&\le 4k + 3 - (k + 1) \\
		&= 3k + 2 \\
		&\le \delta-1.
	\end{align*}
	By Lemma \ref{lem:2.3}, we obtain $|V(G_1^i)| \ge \delta+1$ for $i=1,2$. This completes the proof.
\end{proof}
We now provide the proof of Theorem \ref{thm:1.3}.
\begin{proof}[Proof of Theorem \ref{thm:1.3}]
	For a contradiction, suppose that $G$ does not have property $P(k, \delta)$. By Theorem \ref{thm:1.1}, this implies that the fractional packing number satisfies
	\[ 
	\nu_f(G) \le k + \frac{\delta - 1}{\delta}. 
	\]
	By the definition of $\nu_f(G)$, there exists some partition $\mathcal{P} = \{V_1, V_2, \ldots, V_s\}$ of $V(G)$ with $s \ge 2$ such that
	\begin{equation} \label{eq:4}
		\sum_{1 \le i < j \le s} e(V_i, V_j) \le \left( k + \frac{\delta - 1}{\delta} \right) (s - 1).
	\end{equation}

	Let $r_i = e(V_i, V(G) \setminus V_i)$ for $i = 1, 2, \ldots, s$. Note that the set of edges $\bigcup_{1 \le i < j \le s} E(V_i, V_j)$ forms an edge cut of $G$ that separates $G$ into components containing $G[V_1], G[V_2], \ldots, G[V_s]$. Since $\delta \ge 3k+3$ and $\mu_{n-3}(G) > \frac{9(k+\frac{\delta-1}{\delta})}{\delta+1}$, we have $\mu_{n-3}(G) > \frac{6k}{\delta+1}$. 
	Thus, the condition of Lemma \ref{lem:2.5} is satisfied. Applying Lemma \ref{lem:2.5}, we obtain $r_i \ge k+1$ for all $i = 1, 2, \ldots, s$.
	
	Furthermore, by $\sum_{i=1}^s r_i = 2 \sum_{1 \le i < j \le s} e(V_i, V_j)$ and Inequality \eqref{eq:4}, we have
	\begin{equation} \label{eq:5}
		\sum_{i=1}^{s} r_i \le 2 \left( k + \frac{\delta - 1}{\delta} \right) (s - 1).
	\end{equation}
	Without loss of generality, assume that $r_1 \le r_2 \le \cdots \le r_s$.
	
	Let $t$ be the largest integer such that $r_t < 2 \left( k + \frac{\delta - 1}{\delta} \right).$ We claim that $t \ge 3$.
	
	Indeed, if $t=1$, then 
	\[
	\sum_{i=1}^{s}r_i=r_1+\sum_{i=2}^{s}r_i\ge r_1+2(k+\frac{\delta-1}{\delta})(s-1)>2(k+\frac{\delta-1}{\delta})(s-1),
	\]
	which contradicts Inequality \eqref{eq:5}.
	
	If $t=2$, then 
	\[
	\begin{aligned}
		\sum_{i=1}^{s} r_i
		&= r_1 + r_2 + \sum_{i=3}^{s} r_i \\
		&\ge r_1 + r_2 + 2\!\left(k+\frac{\delta-1}{\delta}\right)(s-2) \\
		&\ge 2(k+1) + 2\!\left(k+\frac{\delta-1}{\delta}\right)(s-2)
		> 2\!\left(k+\frac{\delta-1}{\delta}\right)(s-1),
	\end{aligned}
	\]
	which contradicts Inequality \eqref{eq:5}. Thus, $t \ge 3$ holds.
	
	\noindent\textbf{Case 1. $s\ge t\ge 4$.}

	Since $r_i \ge r_3$ for $3 \le i \le t$ and $r_i \ge 2(k+\frac{\delta-1}{\delta})$ for $i > t$, we obtain
	\[
	r_1 + r_2 + (t-2)r_3 + 2 \left( k + \frac{\delta - 1}{\delta} \right)(s - t) \le \sum_{i=1}^s r_i \le 2 \left( k + \frac{\delta - 1}{\delta} \right)(s - 1). 
	\]
	Using the Inequality $r_3 \ge \frac{r_1+r_2+r_3}{3}$, we have
	\[ 
	r_1 + r_2 + r_3 \le 2 \left( k + \frac{\delta - 1}{\delta} \right)(t - 1) - (t-3)r_3 \le 2 \left( k + \frac{\delta - 1}{\delta} \right)(t - 1) - (t-3)\frac{r_1+r_2+r_3}{3}. 
	\]
	It follows that
	\begin{equation} \label{eq:6}
		r_1 + r_2 + r_3 \le \frac{6(t-1)}{t} \left( k + \frac{\delta - 1}{\delta} \right).
	\end{equation}
	
	Let $V' = V(G) \setminus (V_1 \cup V_2 \cup V_3)$. Then $|V'| \ge \sum_{i=4}^t |V_i| \ge (t-3)(\delta+1).$
	
	Let $r_{ij} = e(V_i, V_j)$ and $r'_i = e(V_i, V')$ for $1\le i, j\le 3$. Consider the quotient matrix $M_4$ of the Laplacian matrix $L(G)$ with respect to the partition $(V_1, V_2, V_3, V')$, which is given by
	\[
	M_4 = 
	\begin{pmatrix}
		\frac{r_1}{|V_1|} & -\frac{r_{12}}{|V_1|} & -\frac{r_{13}}{|V_1|} & -\frac{r'_1}{|V_1|} \\
		-\frac{r_{12}}{|V_2|} & \frac{r_2}{|V_2|} & -\frac{r_{23}}{|V_2|} & -\frac{r'_2}{|V_2|} \\
		-\frac{r_{13}}{|V_3|} & -\frac{r_{23}}{|V_3|} & \frac{r_3}{|V_3|} & -\frac{r'_3}{|V_3|} \\
		-\frac{r'_1}{|V'|} & -\frac{r'_2}{|V'|} & -\frac{r'_3}{|V'|} & \frac{\sum_{i=1}^3 r'_i}{|V'|}
	\end{pmatrix}.
	\]
	
	By Theorem \ref{thm:2.1}, we have $\lambda_i(M_4) \ge \mu_{n-4+i}(G) \ge 0$ for $i=1,2,3,4$.
	By the trace property $\text{tr}(M_4) = \sum_{i=1}^4 \lambda_i(M_4)$, together with Inequality \eqref{eq:6}, we have
	\begin{align*}
		\lambda_1(M_4) \le \text{tr}(M_4)
		&= \frac{r_1}{|V_1|} + \frac{r_2}{|V_2|} + \frac{r_3}{|V_3|} + \frac{r'_1+r'_2+r'_3}{|V'|} \\
		&\le \frac{r_1+r_2+r_3}{\delta+1} + \frac{r_1+r_2+r_3}{(t-3)(\delta+1)} \\
		&= \frac{t-2}{t-3}\cdot\frac{(r_1+r_2+r_3)}{\delta+1} \\
		&\le \frac{6(t-1)(t-2)}{t(t-3)}\cdot\frac{(k+\frac{\delta-1}{\delta})}{\delta+1}.
	\end{align*}
	For $t \ge 4$, it is straightforward to verify that $\frac{(t-1)(t-2)}{t(t-3)} \le \frac{3}{2}$ (equality holds when $t=4$). Consequently,
	\[
	\mu_{n-3}(G) \le \lambda_1(M_4) \le 6 \cdot \frac{3}{2} \cdot \frac{(k+\frac{\delta-1}{\delta})}{\delta+1} = \frac{9(k+\frac{\delta-1}{\delta})}{\delta+1},
	\] 
	which is contrary to the assumption in Theorem \ref{thm:1.3}.
	
	\noindent\textbf{Case 2. $t = 3, s\ge 4$.}
	
	Since $t=3$, we know $r_i \ge 2(k + \frac{\delta-1}{\delta})$ for all $i \ge 4$. So we have
	\[
	\sum_{i=1}^{3}r_i + 2(k + \frac{\delta-1}{\delta})(s-3) \le \sum_{i=1}^{s}r_i \le 2(k + \frac{\delta-1}{\delta})(s-1).
	\]
	This gives
	\begin{equation} \label{eq:7}
		r_1+r_2+r_3 \le 4\left(k + \frac{\delta-1}{\delta}\right).
	\end{equation}
	
	Recall that $r_i \ge k+1$ for all $i = 1, 2, \ldots, s$. If $r_4 \ge 3k+3$, then
	\begin{align*}
		\sum_{i=1}^{s}r_i &= \sum_{i=1}^{3}r_i + \sum_{i=4}^s r_i \\
		&\ge 3(k+1) + (3k+3)(s-3) \\
		&= 3(k+1)(s-2).
	\end{align*}
	But we also have
	\[
	\sum_{i=1}^{s}r_i \le 2(k + \frac{\delta-1}{\delta})(s-1) < 2(k+1)(s-1).
	\]
	So $3(k+1)(s-2) < 2(k+1)(s-1)$, which means $s < 4$. This contradicts $s \ge 4$.
	Therefore, $r_4 \le 3k+2 \le \delta - 1$. By Lemma \ref{lem:2.3}, we have $|V_4| \ge \delta + 1$.
	
	Let $V' = V(G) \setminus (V_1 \cup V_2 \cup V_3)$. Then $|V'| \ge |V_4| \ge \delta+1$.
	
	Let $r_{ij} = e(V_i, V_j)$ and $r'_i = e(V_i, V')$ for $1\le i, j\le 3$. Consider the quotient matrix $M_4$ of the Laplacian matrix $L(G)$ with respect to the partition $(V_1, V_2, V_3, V')$, which is given by
	\[
	M_4 = 
	\begin{pmatrix}
		\frac{r_1}{|V_1|} & -\frac{r_{12}}{|V_1|} & -\frac{r_{13}}{|V_1|} & -\frac{r'_1}{|V_1|} \\
		-\frac{r_{12}}{|V_2|} & \frac{r_2}{|V_2|} & -\frac{r_{23}}{|V_2|} & -\frac{r'_2}{|V_2|} \\
		-\frac{r_{13}}{|V_3|} & -\frac{r_{23}}{|V_3|} & \frac{r_3}{|V_3|} & -\frac{r'_3}{|V_3|} \\
		-\frac{r'_1}{|V'|} & -\frac{r'_2}{|V'|} & -\frac{r'_3}{|V'|} & \frac{\sum_{i=1}^3 r'_i}{|V'|}
	\end{pmatrix}.
	\]
	
	By Theorem \ref{thm:2.1}, we have $\lambda_i(M_4) \ge \mu_{n-4+i}(G) \ge 0$ for $i=1,2,3,4$.
	By the trace property $\text{tr}(M_4) = \sum_{i=1}^4 \lambda_i(M_4)$, together with Inequality \eqref{eq:7}, we have
	\begin{align*}
		\mu_{n-3}(G) \le \lambda_1(M_4) &\le \text{tr}(M_4) 
		= \frac{r_1}{|V_1|} + \frac{r_2}{|V_2|} + \frac{r_3}{|V_3|} + \frac{r'_1+r'_2+r'_3}{|V'|} \\
		&\le \frac{r_1+r_2+r_3}{\delta+1} + \frac{r_1+r_2+r_3}{\delta+1} \\
		&\le 2\cdot \frac{4(k+\frac{\delta-1}{\delta})}{\delta+1} \\
		&= \frac{8(k+\frac{\delta-1}{\delta})}{\delta+1},
	\end{align*}
	which is contrary to the assumption in Theorem \ref{thm:1.3}.
	
	\noindent\textbf{Case 3. $t = 3, s = 3$.}

	Recall from the previous arguments that
	\[ 
	\sum_{i=1}^3 r_i \le 2 \left( k + \frac{\delta - 1}{\delta} \right) (3 - 1) = 4 \left( k + \frac{\delta - 1}{\delta} \right) < 4k+4. 
	\]
	Since $r_i$ must be an integer, we have 
	
	\begin{equation}\label{eq:8}
		r_1 + r_2 + r_3 \le 4k+3.
	\end{equation}
	Also, we have $k+1 \le r_i < 2 \left( k + \frac{\delta - 1}{\delta} \right) < 2k+2 \quad \text{for } i=1,2,3,$ which implies $k+1 \le r_i \le 2k+1$.
	Thus, the conditions of Lemma \ref{lem:4.1} are satisfied. By Lemma \ref{lem:4.1}, there exists an edge subset $X' \subseteq E(G) \setminus \bigcup_{1 \le i < j \le 3} E(V_i, V_j)$ such that $G' = G - (\bigcup_{1 \le i < j \le 3} E(V_i, V_j) \cup X')$ has exactly 4 connected components, which we denote by $U_1, U_2, U_3, U_4$, and $|U_i| \ge \delta+1$ for all $i=1, \dots, 4$.
	
	Let $r'_{ij} = e(U_i, U_j)$ and $r'_i = \sum_{1 \le j \neq i \le 4} r'_{ij}$ in graph $G$ for $1\le i, j\le 4$.
	Note that $|X'| \le \kappa'(G) \le r_1$, $\sum_{i=1}^3 r_i = 2\sum_{1 \le i < j \le 3} e(V_i, V_j)$ and $2r_1 \le \frac{2(r_1+r_2+r_3)}{3}$, we have
	
	\begin{equation}\label{eq:9}
		\sum_{i=1}^4 r'_i = 2 \left( \sum_{1 \le i < j \le 3} e(V_i, V_j) + |X'| \right) \le \sum_{i=1}^3 r_i + 2r_1 \le 3r_1 + r_2 + r_3 \le \frac{5(4k+3)}{3}. 
	\end{equation}

	Now, consider the quotient matrix $M'_4$ of $L(G)$ with respect to the partition $\{U_1, U_2, U_3, U_4\}$,
	\[
	M'_4 = 
	\begin{pmatrix}
		\frac{r'_1}{|U_1|} & -\frac{r'_{12}}{|U_1|} & -\frac{r'_{13}}{|U_1|} & -\frac{r'_{14}}{|U_1|} \\
		-\frac{r'_{12}}{|U_2|} & \frac{r'_2}{|U_2|} & -\frac{r'_{23}}{|U_2|} & -\frac{r'_{24}}{|U_2|} \\
		-\frac{r'_{13}}{|U_3|} & -\frac{r_{23}}{|U_3|} & \frac{r'_3}{|U_3|} & -\frac{r'_{34}}{|U_3|} \\
		-\frac{r'_{14}}{|U_4|} & -\frac{r'_{24}}{|U_4|} & -\frac{r'_{34}}{|U_4|} & \frac{r'_4}{|U_4|}
	\end{pmatrix}.
	\]
	By Theorem \ref{thm:2.1}, we have $\lambda_i(M'_4) \ge \mu_{n-4+i}(G) \ge 0$ for $i=1,2,3,4$.
	By the trace property $\text{tr}(M'_4) = \sum_{i=1}^4 \lambda_i(M'_4)$, together with Inequality \eqref{eq:8} and Inequality \eqref{eq:9}, we have
	\begin{align*}
		\mu_{n-3}(G) \le \lambda_1(M'_4) \le \text{tr}(M'_4) 
		&= \sum_{i=1}^4 \frac{r'_i}{|U_i|} 
		\le \frac{\sum_{i=1}^4 r'_i}{\delta+1} \\
		&\le \frac{3r_1 + r_2 + r_3}{\delta+1} \\
		&\le \frac{5(4k+3)}{3(\delta+1)} 
		< \frac{9(k+\frac{\delta-1}{\delta})}{\delta+1},
	\end{align*}
	 which is contrary to the assumption in Theorem \ref{thm:1.3}.

	This completes the proof.
\end{proof}

Similar to Corollary \ref{cor:3.1}, we extend Theorem \ref{thm:1.3} to the general matrices $aD(G)+A(G)$ and $aD(G)+bA(G)$.
\noindent\begin{corollary} \label{cor:4.2}
	Let $k \ge 2$ be an integer and let $G \in \mathcal{G}_2$ be a graph with minimum degree $\delta \ge 3k+3$. Then $G$ has property $P(k, \delta)$ if one of the following conditions holds:
	\begin{enumerate}[{\rm(i)}]
		\item $\lambda_4(G, a) < (a+1)\delta - \frac{9(k+\frac{\delta-1}{\delta})}{\delta+1}$;
		\item If $b > 0$, $\lambda_4(G, a, b) < (a+b)\delta - \frac{9b(k+\frac{\delta-1}{\delta})}{\delta+1}$;
		\item If $b < 0$, $\lambda_{n-3}(G, a, b) > (a+b)\delta - \frac{9b(k+\frac{\delta-1}{\delta})}{\delta+1}$.
	\end{enumerate}
\end{corollary}
\begin{proof}
	For a contradiction, suppose that $G$ does not have property $P(k, \delta)$. By Theorem \ref{thm:1.3}, we have $\mu_{n-3}(G) \le \frac{9(k+\frac{\delta-1}{\delta})}{\delta+1}$.
	
	By Weyl's Inequalities (Theorem \ref{thm:2.2}(ii)), we have
	\begin{align*}
		\lambda_4(G, a) 
		&\ge \lambda_n((a+1)D(G)) + \lambda_4(-L(G)) \\
		&= (a+1)\delta - \mu_{n-3}(G) \\
		&\ge (a+1)\delta - \frac{9(k+\frac{\delta-1}{\delta})}{\delta+1},
	\end{align*}
	which contradicts condition (i). Thus, (i) holds.
	
	For the general case $b \neq 0$, note that $\lambda_i(G, a, b) = b \lambda_i(G, \frac{a}{b})$ if $b > 0$, and $\lambda_{n-i+1}(G, a, b) = b \lambda_i(G, \frac{a}{b})$ if $b < 0$. Conditions (ii) and (iii) then follow immediately from (i) by appropriate substitution and scaling.
\end{proof}

Note that $\lambda_4(G, 0, 1) = \lambda_4(G)$ and $\lambda_4(G, 1, 1) = q_4(G)$. So we obtain the following corollary by Corollary \ref{cor:4.2} directly.
\noindent\begin{corollary}\label{cor:4.3} 
	Let $k \ge 2$ be an integer and let $G \in \mathcal{G}_2$ be a graph with minimum degree $\delta \ge 3k+3$. The following statements hold:
	\begin{enumerate}[{\rm(i)}]
		\item If $\lambda_{4}(G)<\delta-\frac{9(k+\frac{\delta-1}{\delta})}{\delta+1}$, then $G$ has property $P(k, \delta)$;
		\item If $q_{4}(G)<2\delta-\frac{9(k+\frac{\delta-1}{\delta})}{\delta+1}$, then $G$ has property $P(k, \delta)$.
	\end{enumerate}
\end{corollary}

\section{Extensions of the Spectral Condition on $\lambda_{2}(G)$}
In this section, we extend the results of Cai and Zhou $\cite{RF10}$ regarding the second largest adjacency eigenvalue $\lambda_2(G)$ to the matrices of the forms $aD(G) + A(G)$ and $aD(G) + bA(G)$, where $a\ge 0$ and $b \neq 0$. We begin by recalling their main result.	

\noindent\begin{theorem}[Theorem 1.6 in \cite{RF10}] \label{thm:5.1}
	Let $k$ be a positive integer, and let $G$ be a graph with minimum degree $\delta \ge 2k + 2$. If
	\[
	\lambda_2(G) < \delta - \frac{2(k + \frac{\delta-1}{\delta})}{\delta + 1},
	\]
	then $G$ has property $P(k, \delta)$.
\end{theorem}

Based on Theorem \ref{thm:5.1}, we obtain the following results.

\noindent\begin{corollary} \label{cor:5.2}
	Let $k$ be a positive integer and let $G$ be a graph with minimum degree $\delta \ge 2k+2$ and maximum degree $\Delta$. Let $a \ge 0$ and $b \neq 0$ be real numbers. Then $G$ has property $P(k, \delta)$ if one of the following conditions holds:
	\begin{enumerate}[{\rm(i)}]
		\item $\lambda_2(G, a) < (a+1)\delta - \frac{2(k+\frac{\delta-1}{\delta})}{\delta+1}$;
		\item If $b > 0$, $\lambda_2(G, a, b) < (a+b)\delta - \frac{2b(k+\frac{\delta-1}{\delta})}{\delta+1}$;
		\item If $b < 0$, $\lambda_{n-1}(G, a, b) > a\Delta + b\delta - \frac{2b(k+\frac{\delta-1}{\delta})}{\delta+1}$.
	\end{enumerate}
\end{corollary}

\begin{proof}
	Suppose to the contrary that $G$ does not have property $P(k, \delta)$. By Theorem \ref{thm:5.1}, we have $\lambda_2(G) \ge \delta - \frac{2(k+\frac{\delta-1}{\delta})}{\delta+1}$. 
	
	Consider the matrix $aD(G) + A(G)$. By Weyl's Inequalities (Theorem \ref{thm:2.2}(ii)), we have
	\begin{align*}
		\lambda_2(G, a) = \lambda_2(aD + A) &\ge \lambda_n(aD) + \lambda_2(A) \\
		&= a\delta + \lambda_2(G) \\
		&\ge a\delta + \delta - \frac{2(k+\frac{\delta-1}{\delta})}{\delta+1} \\
		&= (a+1)\delta - \frac{2(k+\frac{\delta-1}{\delta})}{\delta+1},
	\end{align*}
	which contradicts condition (i). Thus, (i) holds.
	
	Note that $aD(G) + bA(G) = b(\frac{a}{b}D(G) + A(G))$. If $b > 0$, we have $\lambda_2(G, a, b) = b\lambda_2(G, \frac{a}{b})$. Condition (ii) follows directly from (i) by substituting $a$ with $\frac{a}{b}$ and multiplying the inequality by $b$.
	
	Consider the case $b < 0$. By Weyl's Inequalities (Theorem \ref{thm:2.2}(i)), we have
	\begin{align*}
		\lambda_{n-1}(G, a, b) = \lambda_{n-1}(aD + bA) &\le \lambda_1(aD) + \lambda_{n-1}(bA) \\
		&= a\Delta + b\lambda_2(G) \\
		&\le a\Delta + b\delta - \frac{2b(k+\frac{\delta-1}{\delta})}{\delta+1},
	\end{align*}
	which contradicts condition (iii). Thus, (iii) holds.
\end{proof}

Note that $\lambda_2(G, 1) = q_2(G)$ and $\lambda_{n-1}(G, 1, -1) = \mu_{n-1}(G)$. So we obtain the following corollary by Corollary \ref{cor:5.2} directly.

\noindent\begin{corollary} \label{cor:5.3}
	Let $k$ be a positive integer and let $G$ be a graph with minimum degree $\delta \ge 2k+2$ and maximum degree $\Delta$. The following statements hold:
	\begin{enumerate}[{\rm(i)}]
		\item If $q_2(G) < 2\delta - \frac{2(k+\frac{\delta-1}{\delta})}{\delta+1}$, then $G$ has property $P(k, \delta)$;
		\item If $\mu_{n-1}(G) > \Delta - \delta + \frac{2(k+\frac{\delta-1}{\delta})}{\delta+1}$, then $G$ has property $P(k, \delta)$.
	\end{enumerate}
\end{corollary}

\section*{Declaration of competing interest}
\quad\quad The authors declare that they have no known competing financial interests or personal relationships that could have appeared to influence the work reported in this paper.

\section*{Data availability}
\quad\quad No data was used for the research described in the article.

\end{document}